\documentclass[a4paper,10pt]{amsart}
\usepackage[T1]{fontenc}
\usepackage{amsmath,amssymb,amsfonts}
\usepackage[fleqn,tbtags]{mathtools}
\usepackage{bbm}
\usepackage{epsfig}
\usepackage{psfrag}
\usepackage[all,knot]{xy}

\newtheorem{theorem}{\bfseries Theorem}[section]
\newtheorem{lemma}[theorem]{\bfseries Lemma}
\newtheorem{proposition}[theorem]{\bfseries Proposition}

\newtheorem{corollary}[theorem]{\bfseries Corollary}
\newtheorem*{theoreminv}{{\bfseries Theorem} \ref{thm:inv}}
\newtheorem*{theoremmax}{{\bfseries Theorem} \ref{thm:maxsplit}}
\newtheorem*{cor1}{{\bfseries Corollary} \ref{cor:maxHTV}}
\theoremstyle{definition}

\newcommand{\C}{\mathcal{C}}
\newcommand{\pt}{\otimes}
\newcommand{\m}{\mathcal}
\newcommand{\ov}[1]{\overline{#1}}

\newcommand{\fd}{\rightarrow}
\newcommand{\Cc}{\mathbb{C}}
\newcommand{\Bbbi}{\mathbbm{1}}
\newcommand{\id}{\textrm{id}}
\newcommand{\xd}[1]{#1^{\vee}}
\newcommand{\xg}[1]{{}^{\vee}#1}
\newcommand{\xdd}[1]{#1^{\vee \vee}}
\newcommand{\ap}{\mapsto}

\newcommand{\tr}{\mathrm{tr}}
\newcommand{\ho}[2]{Hom_{\C}(#1,#2)}
\newcommand{\lc}{\Lambda_{\mathcal{C}}}
\newcommand{\grad}{\Gamma_{\mathcal{C}}}
\newcommand{\cs}[1]{|#1|}
\def\commutatif{\ar@{}[dll]|{\circlearrowleft}}
\newcommand{\Zz}{\mathbb{Z}}
\newcommand{\dis}[1]{\displaystyle{#1}}
\newcommand{\dc}{\Delta_{\C}}
\newcommand{\im}{\textrm{im}}

\begin{document}

\markboth{J\'er\^ome Petit}
{The maximal decomposition of the Turaev-Viro TQFT.}

\title{The maximal decomposition of the Turaev-Viro TQFT}
\author{J\'er\^ome Petit}
\address{Department of Mathematics, Tokyo institute of technology, 2-12-1 Ookayama Meguro-ku, 152-8551 Tokyo, Japan\\ petit.j.aa@m.titech.ac.jp}

\begin{abstract}
In \cite{HQFTp} we have built an homotopical Turaev-Viro invariant and an HQFT from the universal graduation of a spherical category. In the present paper, we show that every graduation $(G,p)$ of a spherical category $\C$ defines an homotopical Turaev-Viro invariant $HTV_{\C}^{(G,p)}$  and an HQFT $\m{H}_{\C}^{(G,p)}$. Furthermore we show that the Turaev-Viro TQFT will be split into blocks coming the HQFT $\m{H}_{\C}^{(G,p)}$. We show that this decomposition is maximal for the universal graduation of the category, which means that for every graduation $(G,p)$ the HQFT $\m{H}_{\C}^{(G,p)}$ is split into blocks coming from the HQFT obtained from the universal graduation.
\end{abstract}

\maketitle

\section{Introduction}

The Turaev-Viro invariant \cite{TV} is a \emph{quantum invariant} of 3-manifold with boundary. In the original construction, Turaev and Viro used the quantum group $U_q(\frak{sl}_2)$ to build this invariant. In \cite{BW} and \cite{GK}, the authors generalize the construction to spherical categories with invertible dimension. A spherical category is a semisimple sovereign category over a commutative ring $\Bbbk$ such that the left and right traces coincide. The dimension of a spherical category is the sum of squares of dimensions of simple objects. The Turaev-Viro invariant of a closed 3-manifold $M$ is a state-sum indexed by the colorings of a triangulation of $M$. The colorings of a triangulation $T$ are maps from the set of oriented 1-simplexes to the set of scalar objects (up to isomorphism) of a spherical category $\C$. The set of colorings of a triangulation $T$ is denoted $Col(T)$. The Turaev-Viro invariant is~:
$$
TV_{\C}(M)=\dc^{-n_0(T)}\sum_{c\in Col(T)}w_cW_c\in \Bbbk\, ,
$$
where $\dc$ is the dimension of the category, $n_0(T)$ is the number of 0-simplexes of $T$, $w_c$ is a scalar obtained from the coloring of the 1-simplexes and the trace of the category and $W_c$ is a scalar obtained from the 6j-symbols of the category.

The Turaev-Viro invariant extends to a \emph{Topological Quantum Field Theory} (TQFT) \cite{Tu} called \emph{Turaev-Viro TQFT}. In dimension 2+1, a TQFT assigns to every closed surface a finite dimensional vector space and to every cobordism a linear map. In \cite{HQFTp}, we show that the Turaev-Viro TQFT comes from an HQFT. \emph{An Homotopical Quantum Field Theory \cite{THQFT}} (HQFT) is a TQFT  for surfaces and cobordisms endowed with homotopy classes of continuous map to a \emph{target space} $X$. Roughly speaking, the Turaev-Viro TQFT splits into blocks coming from an HQFT \cite{HQFTp}. To obtain this decomposition, we assign to each spherical category $\C$ a group $\grad$ which comes from the universal graduation $(\grad,\cs{?})$ of the category. A graduation of a semisimple tensor category is a pair $(G,p)$ where $G$ is a group and $p$ is a map from $G$ to the set of isomorphisms classes of scalar objects such that $p(Z)=p(X)p(Y)$ if $Z$ is a scalar subobject of $X\pt Y$. Using the group $\grad$ we define an homotopical invariant $HTV_{\C}$ called \emph{the homotopical Turaev-Viro invariant}. This homotopical invariant will split the Turaev-Viro invariant. More precisely, we observe that every coloring $c$ of a triangulation $T$ of a closed 3-manifold $M$ leads to an homotopy class $x_c\in [M,B\grad]$, where $B\grad$ is the classifying space of the group $\grad$ and $[M,B\grad]$ is the set of homotopy classes of continuous map from $M$ to $B\grad$. These remarks lead to the following homotopical invariant of closed 3-manifolds~:

$$
HTV_{\C}(M,x)=\dc^{-n_0(T)}\sum_{\substack{c\in Col(T)\\x_c=x}}w_cW_c\, ,
$$
where $x\in [M,B\grad]$.   In \cite{HQFTp} , we show that the homotopical Turaev-Viro extends to an HQFT with target space $B\grad$ denoted $\m{H}_{\C}$ and we obtain the following decomposition of the Turaev-Viro TQFT $\m{V}_{\C}$:
$$
\m{V}_{\C}(\Sigma)=\bigoplus_{x\in [\Sigma, B\grad]}\m{H}_{\C}(\Sigma,x)\, ,
$$
for every closed and oriented surface $\Sigma$.

The motivation for this paper was to study the maximality of this decomposition. More precisely, we want to define other decomposition of the Turaev-Viro TQFT and compare them. To fulfill this objective, we will associate to every graduation of a spherical category an HQFT with target space the classifying space of the graduation. This HQFT will give the block of the decomposition of the Turaev-Viro TQFT and it will be an extension of an homotopical invariant associated to a graduation of $\C$. More precisely, for every graduation $(G,p)$ of $\C$ we build an homotopical Turaev-Viro invariant and this invariant will split the Turaev-Viro invariant:

\begin{theoreminv}
Let $\C$ be a spherical category with an invertible dimension, $M$ be 3-manifold, $\Sigma$ be the boundary of $M$ and $T_0$ be a triangulation of $\Sigma$. For every coloring $c_0\in Col(T_0)$ and for every homotopy class $x\in [M,BG]_{\Sigma,x_{c_0}}$, where $x_{c_0}\in [\Sigma,BG]$ is obtained from $c_0$, the vector :
$$HTV^{(G,p)}_{\C}(M,c_0,x)=\dc^{-n_0(T)+n_0(T_0)/2}\sum_{c\in Col_{c_0,x}(T)}w_cW_c \in V_{\C}(\Sigma,T_0,c_0)\,$$
is an invariant of the triple $(M,x,c_0)$. We have the following equality:
\begin{equation}\label{TVsplit}
TV_{\C}(M,c_0)=\sum_{x\in [M,BG]_{\Sigma,x_{c_0}}}HTV^{(G,p)}_{\C}(M,c_0,x)\, .
\end{equation}
\end{theoreminv}

Using the universal property of the universal graduation, we can compare the decompositions of the Turaev-Viro invariant obtained from a graduation $(G,p)$ and from the universal graduation. The universal property of the universal graduation induces a group morphism $f : \grad \fd G$. The group morphism $f$ induced a map $F : [M,B\grad] \fd [M,BG]$, using this map we show that for every graduation $(G,p)$ the homotopical Turaev-Viro invariant $HTV_{\C}^{(G,p)}$ comes from the homotopical Turaev-Viro invariant $HTV_{\C}^{(\grad,\cs{?})}$:

\begin{cor1}
Let $\C$ be a spherical category with an invertible dimension, $M$ be a 3-manifold, $\Sigma$ be the boundary of $M$ and $T_0$ be a triangulation of $\Sigma$. For every graduation $(G,p)$ of $\C$, one gets:
$$
TV_{C}(M,c_0)=\sum_{x\in [M,BG]}HTV_{\C}^{(G,p)}(M,x,c_0)\in V_{\C}(\Sigma,T_0,c_0)\,
$$
with $c_0\in Col(T_0)$, and
$$
HTV_{\C}^{(G,p)}(M,x,c_0)=\sum_{y\in F^{-1}(x)}HTV^{(\grad,\cs{?})}(M,y,c_0)\,,
$$
where $F$ is the map induced by the universal graduation $(\grad,\cs{?})$.
\end{cor1}

For every graduation $(G,p)$ of $\C$, we prove that the homotopical invariant $HTV_{\C}^{(G,p)}$ extends to an HQFT $\m{H}_{\C}^{(G,p)}$ with target space $BG$ and we show that for every graduation $(G,p)$ the HQFT  $\m{H}_{\C}^{(G,p)}$ will split the Turaev-Viro TQFT:
$$
\m{V}_{\C}(\Sigma)=\bigoplus_{x\in [\Sigma,BG]}\m{H}^{(G,p)}_{\C}(\Sigma)\, ,
$$
for every closed and oriented surface $\Sigma$. Using Theorem \ref{thm:inv} and Corollary \ref{cor:maxHTV}, we show that the decomposition of the Turaev-Viro TQFT given by the universal graduation is maximal:

\begin{theoremmax}
Let $\C$ be  a spherical category, $(G,p)$ be a graduation of $\C$. The Turaev-Viro HQFT obtained from the graduation $(G,p)$ is decomposed in the following way:
$$
\m{V}^{(G,p)}_{\C}(\Sigma,x)=\bigoplus_{y\in F^{-1}(x)}\m{V}^{(\grad,\cs{?})}_{\C}(\Sigma,y)\, ,
$$
for every closed surface $\Sigma$, and for every $x\in [\Sigma,BG]$, the map $F : [\Sigma,B\grad]\fd [\Sigma,BG]$ is the map obtained from the universal graduation (Lemma \ref{lem:mapgrad}).
\end{theoremmax}

The rest of the paper is organized as follows. In Section \ref{sec:category}, we review several facts about
monoidal categories and we define the universal graduation of semisimple tensor categories. In Section \ref{sec:TVconstruc}, we recall the construction of the Turaev-Viro invariant. In Section \ref{sec:TVhom}, we will build an homotopical Turaev-Viro invariant for every graduation $(G,p)$ of  a spherical category $\C$. Furthermore we show that the  homotopical Turaev-Viro invariant $HTV_{\C}^{(G,p)}$ splits the Turaev-Viro invariant (Theorem \ref{thm:inv}). In Section \ref{sec:TVhommax}, we compare the different splittings of the Turaev-Viro invariant. We show that The Turaev-Viro invariant and the homotopical Turaev-Viro invariant are obtained from the homotopical Turaev-Viro invariant $HTV_{\C}^{(\grad,\cs{?})}$ (Corollary \ref{cor:maxHTV}). In Section \ref{sec:TVHQFT}, we extend the homotopical invariant build forma graduation to an HQFT. The target of this HQFT will be the classifying space of the graduation. In Section \ref{sec:TVHQFTmax}, we prove Theorem \ref{thm:maxsplit}, it follows that the Turaev-Viro TQFT and the Turaev-Viro HQFT obtained from any graduation of $\C$ are decomposed into blocks which come from the Turaev-Viro HQFT obtained for the universal graduation.

\subsection*{Notations and conventions}
Throughout this paper, $\Bbbk$ will be a commutative, algebraically closed and characteristic zero field. Unless otherwise specified, categories are assumed to be small and monoidal categories are assumed to be strict and spherical categories are assumed to be strict.

Throughout this paper, we use the following notation. For an oriented manifold $M$, we denote by $\overline{M}$ the same manifold with the opposite orientation.

\section{Graduations of tensor categories}\label{sec:category}

In the present Section, we review a few general facts about categories with structure, which we use intensively throughout this text.

Let $\C$ be a monoidal category. A \emph{duality} of $\C$ is a data $(X,Y,e,h)$, where $X$ and $Y$ are objects of $\C$ and $e : X\pt Y \fd \Bbbi$ (\emph{evaluation}) and  $h : \Bbbi \fd Y\pt X$ (\emph{coevaluation}) are morphisms of $\C$, satisfying:
$$
(e\pt \id_X)(\id_X\pt h)=\id_X\qquad \mbox{and}\qquad (\id_Y\pt e)(h\pt \id_Y)=\id_Y\,.
$$
If $(X,Y,e,h)$ is a duality, we say that $(Y,e,h)$ is \emph{a right dual of $X$}, and $(X,e,h)$ is \emph{a left dual of $Y$}. If a right or left dual of an object exists, it is unique up to unique isomorphism.

A \emph{right autonomous} (resp. \emph{left autonomous}, resp. \emph{autonomous}) category is a monoidal category for which every object admits a right dual (resp. a  left dual, resp. both a left and a right dual).

If $\C$ has right duals, we may pick a right dual $(\xd{X},e_X,h_X)$ for each object $X$. This defines a monoidal functor  $\xdd{?} : \C\fd \C$ defined by $X\ap \xdd{X}$ and $f\ap \xdd{f}$, called the \emph{double right dual functor}.

\medskip

\subsection*{Sovereign categories}

A \emph{sovereign structure} on a right autonomous category $\C$ consists in the choice of a right dual for each object of $\C$ together with a monoidal isomorphism $\phi : 1_{\C} \fd \xdd{?}$, where $1_{\C}$ is the identity functor of $\C$. Two sovereign structures are \emph{equivalent} if the corresponding monoidal isomorphisms coincide via the canonical identification of the double dual functors.

A \emph{sovereign category} is a right autonomous category endowed with an equivalence class of sovereign structures.

Let $\C$ be a sovereign category, with chosen right duals $(\xd{X},e_X,h_X)$ and sovereign isomorphisms $\phi_X : X\fd \xdd{X}$. For each object $X$ of $\C$, we set :

$$
\epsilon_X = e_{\xd{X}}(\id_{\xd{X}}\pt \phi_X) \qquad \textrm{and} \qquad \eta_X = (\phi_X^{-1}\pt\id_{\xd{X}})h_{\xd{X}}\, .
$$

Then $(\xd{X},\epsilon_X,\eta_X)$ is a left dual of $X$. Therefore $\C$ is autonomous. Moreover the right left functor $\xg{?}$ defined by this choice of left duals coincides with $\xd{?}$ as a monoidal functor. From now on, for each sovereign category $\C$ we will make this choice of duals.

The sovereign categories are an appropriate categorical setting for a good notion of trace. Let $\C$ be a sovereign category and $X$ be an object of $\C$. For each endomorphism $f\in Hom_{\
C}(X,X)$,
$$\tr_l(f)=\epsilon_X (\id_{\xd{X}}\pt f)h_X \in Hom_{\C}(\Bbbi,\Bbbi)$$
is the \emph{left trace} of $f$ and
$$\tr_r(f)=e_X(f\pt \id_{\xd{X}})\eta_X\in Hom_{\C}(\Bbbi,\Bbbi)$$
is the \emph{right trace} of $f$. We denote by $\dim_r(X)=\tr_r(\id_X)$ (resp. $\dim_l(X)=\tr_l(\id_X)$) the \emph{right dimension} (resp.
\emph{left dimension}) of $X$.

\subsection*{Tensor categories}

By a \emph{$\Bbbk$-linear category}, we shall mean a category for which the set of morphisms are $\Bbbk$-spaces, the composition is $\Bbbk$-bilinear
, there exists a null object and for every objects $X$, $Y$ the direct sum $X\oplus Y$ exists in $\C$.

A $\Bbbk$-linear category is \emph{abelian} if it admits finite direct sums, every morphism has a kernel and a cokernel, every monomorphism is the kernel of its cokernel, every epimorphism is the cokernel of its kernel, and every morphism is expressible as the composite of an epimorphism followed by a monomorphism.

An object $X$ of an abelian $\Bbbk$-category $\C$ is \emph{scalar} if $Hom_{\C}(X,X)\cong \Bbbk$.

A \emph{tensor category over $\Bbbk$} is an autonomous category endowed with a structure of $\Bbbk$-linear abelian category such that the tensor product is $\Bbbk$-bilinear and the unit object is a scalar object.

A $\Bbbk$-linear category is \emph{semisimple} if :
\begin{itemize}
\item[(i)] every object of $\C$ is a finite direct sum of scalar objects,
\item[(ii)] for every scalar objects $X$ and $Y$, we have : $X\cong Y$ or $\ho{X}{Y}=0$\,.
\end{itemize}

A \emph{semisimple $\Bbbk$-category} is a semisimple abelian $\Bbbk$-linear category and every simple object is a scalar object. Notice that in a semisimple abelian $\Bbbk$-category every scalar object is a simple object. By a \emph{finitely semisimple $\Bbbk$-category} we shall mean a semisimple $\Bbbk$-category which has finitely many isomorphism classes of scalar objects. The set of isomorphism classes of scalar objects of an abelian $\Bbbk$-category $\C$ is denoted by $\lc$.

%

\subsection*{Graduations}

Let $\C$ be semisimple tensor $\Bbbk$-category and $G$ be a group. A \emph{$G$-graduation of $\C$} is a map $p : G\fd \lc$ satisfying :
\begin{itemize}
\item[$\bullet$] $p(Z)=p(X)p(Y)$, for every scalar objects $X,Y,Z$ such that $Z$ is a subobject of $X\pt Y$.
\end{itemize}

A \emph{graduation of $\C$} is a pair $(G,p)$, where $G$ is group and $p$ is a $G$-graduation of $\C$.

By induction, the multiplicity property of a graduation can be extended to $n$-terms.

In \cite{HQFTp}, we prove that every semisimple tensor $\Bbbk$-category admits an universal graduation:
\begin{proposition}[\cite{HQFTp}]\label{pro:graduation}
Let $\C$ be a semisimple tensor $\Bbbk$-category. There exists a graduation $(\grad,\cs{?})$ of $\C$ satisfying the following universal property : for every graduation $(G,p)$ of $\C$, there exists a unique group morphism $f : \grad \fd G$ such that the diagram:
$$
\xymatrix{\lc \ar[rr]^{\cs{?}} \ar[dr]_{p} && \grad \ar[dl]^{f}\commutatif\\
&G& }
$$
commutes.
\end{proposition}

Let $\C$ be a semisimple tensor $\Bbbk$-category, the group $\grad$ which defines the universal graduation $(\grad,\cs{?})$ is called \emph{the graduator of $\C$}. The graduator can be used to describe the sovereign (resp. spherical) structures of a sovereign (resp. spherical) category \cite{HQFTp}.

\subsubsection*{Example} The graduator of the category of representations of finite dimension of $U_q(\frak{sl}_n)$ is $\Zz_n$

\subsection*{Spherical categories}A \emph{spherical category} is a sovereign, finitely semisimple tensor $\Bbbk$-category satisfying:
\begin{itemize}
\item[$\bullet$] for every object $X$ of $\C$ and for every morphism $f : X \fd X$ : $\tr_r(f)=\tr_l(f)$.
\end{itemize}

A \emph{spherical structure} on $\C$ is a sovereign structure on $\C$ such that $\C$ is a spherical category.

From now on, for every spherical category the left and right trace (resp. dimension) will be denoted by $\tr$ (resp. $\dim$).

The dimension of a spherical category is the scalar : $\dis{\dc=\sum_{X\in \lc}\dim(X)^2}\in \Bbbk$. From now on, unless otherwise specified, spherical categories are assumed to have an invertible dimension.

\section{The Turaev-Viro invariant}\label{sec:TVconstruc}

In this Section, we recall the construction of the Turaev-Viro invariant. For further reading on the Turaev-Viro invariant, we refer the reader to \cite{TV} (the original construction), \cite{BW} (the construction using a spherical category), \cite{GK} and \cite{Tu}. Throughout this Section, $\C$ will be a spherical category.

\medskip

An \emph{orientation} of a $n$-simplex $F$ is a map  $o : Num(F)\fd \{\pm1\}$, where $Num(F)$ is the set of numberings of $F$, invariant under the action of the alternated group $\mathfrak{A}_{N+1}\subset \mathfrak{S}_{N+1}$.

Let $T$ be an oriented simplicial complex, we denote the set of oriented $p$-simplexes by $T^p_o$. A \emph{coloring} of $T$ is a map $c : T^1_o \fd \lc$ satisfying : \begin{itemize}
\item[(i)] $c(x_1x_2)=c(x_2x_1)^{\vee}$, for every oriented 1-simplex $(x_1x_2)$,
\item[(ii)] the unit object $\Bbbi$ is a subobject of $c(x_1x_2)\pt c(x_2x_3)\pt c(x_3x_1)$ for every oriented 2-simplex $(x_1x_2x_3)$.
\end{itemize}
We denote by $Col(T)$ the set of colorings of $T$ .

Let $f$ be an oriented 2-simplex, $c$ be a coloring of $T$ and $\nu=(x_1x_2x_3)$ be a numbering of $f$ compatible with the orientation of $f$. Set :
$$
V_{\C}(f,c)_{\nu}=\ho{\Bbbi}{c(x_1x_2)\pt c(x_2x_3)\pt c(x_3x_1)}\, .
$$

The vector space $V_{\C}(f,c)$ does not depend on the choice of the numbering compatible with the orientation (e.g. \cite{BW}, \cite{GK}, \cite{Tu}). From now on, the vector space $V_{\C}(f,c)_{\nu}$, with $\nu=(x_1x_2x_3)$, will be denoted by $V_{\C}(x_1x_2x_3,c)$. If there is no ambiguity on the choice of the coloring $c$, then $V_{\C}(x_1x_2x_3,c)$ will be denoted by $V_{\C}(x_1x_2x_3)$.

Let us recall some properties of the vector space defined above. For every scalar objects $X$, $Y$ and $Z$, we set:
\begin{align}
\omega_{\C} : Hom_{\C}(\Bbbi,X\pt Y \pt Z)\pt_{\Bbbk} Hom_{\C}(\Bbbi,\xd{Z}\pt \xd{Y} \pt \xd{Z})&\fd \Bbbk^* \label{bilinear}\\
f\pt g & \ap tr(\xd{f}g)\, .\nonumber
\end{align}
For every spherical category $\C$, the bilinear form $\omega_{\C}$ is non degenerate (e.g. \cite{BW}, \cite{GK}, \cite{Tu}). Let $f$ be an oriented 2-simplex, we denote by $\ov{f}$ the 2-simplex $f$ endowed with the opposite orientation. Let $c$ be a coloring of $f$, the bilinear form (\ref{bilinear}) induces: $V_{\C}(f,c)^{*}\cong V_{\C}(\ov{f},c)$.

In the construction of the Turaev-Viro invariant, we assign to every oriented 3-simplex of a 3-manifold $M$, endowed with a coloring, a vector which lies in the vector space defined by the faces of the 3-simplex. The vector assigned to each 3-simplex is obtained by the 6j-symbols of the category. A contraction of these vectors along the 2-simplexes contained inside the 3-manifold $M$ leads to a scalar if the manifold $M$ is without boundary or to a vector in $\dis{\bigotimes_{f\in T^2_{\partial M} }V_{\C}(f,c)}$ if the manifold $M$ has a boundary $\partial M$. We denote this vector (or scalar)  by $W_c$, for every coloring $c$.

  We introduce some notations. Let $\Sigma$ be an oriented closed surface endowed with a triangulation $T_0$. For every coloring $c_0$  of $T_0$, we set : $\dis{V_{\C}(\Sigma,T_0,c_0)=\bigotimes_{f\in T^2}V_{\C}(f,c_0)}$ and $\dis{V_{\C}(\Sigma,T_0)=\bigoplus_{c\in Col(T_0)}V_{\C}(\Sigma,T_0,c)}$. Let $M$ be 3-manifold with boundary $\Sigma$ and $T$ be a triangulation of $M$ such that its restriction to $\Sigma$ is $T_0$. For every coloring $c_0\in Col(T_0)$, we denote by $Col_{c_0}(T)$ the set of colorings of $T$ such that the restriction to $T_0$ is $c_0$. With this notation, for every coloring $c\in Col_{c_0}(T)$, we have: $W_c\in V_{\C}(\Sigma,T_0,c_0)$. Furthermore we choose a square root $\dc^{1/2}$ of $\dc$.

For every scalar object $X$ of $\C$, we set $\dim(X)^{1/2}$ a square root of $\dim(X)$. The equalities $\dim(X)^{1/2}= \dim(\xd{X})^{1/2}$ and $\dim(X)= \dim(\xd{X})$ ensure independence of $\dim(c(e))$, $\dim(c(e))^{1/2}$ of the choice of the orientation of $e$, for every coloring $c$.

\begin{theorem}[Turaev-Viro invariant \cite{BW}, \cite{GK}, \cite{Tu}, \cite{TV} ]
Let $\C$ be a spherical category with an invertible dimension, $M$ be a compact oriented 3-manifold and $\partial M$ be the boundary of $M$ endowed with a triangulation $T_0$. For every coloring $c_0\in Col(T_0)$, we set :
\begin{equation}
\label{invTV}
\scriptsize{TV_{\C}(M,c_0)=\dc^{-n_0(T)+n_0(T_0)/2}\sum_{c \in Col_{c_0}(T)}\prod_{e\in T^1_0}\dim(c_0(e))^{1/2}\prod_{e\in
T^1\backslash T^1_0}\dim(c(e))W_{c}\in V(\partial M,c_0T_0)\, ,}
\end{equation}
where $n_0(T)$ (resp. $n_0(T_0)$) is the number of $0$-simplexes of $T$ (resp. $T_0$) and $T^1\backslash T^1_0$ is the set of 1-simplexes of $M\backslash\partial M$.
For every coloring $c_0\in Col(T_0)$, the vector $TV_{\C}(M,c_0)$ is independent on the choice of the triangulation of $M$ which extends $T_0$. The Turaev-Viro invariant is the vector:
$$
TV_{\C}(M)=\sum_{c_0\in Col(T_0)}TV_{\C}(M,c_0)\in V_{\C}(\partial M,T_0)=\bigoplus_{c_0\in Col(T_0)}V(\partial M, T_0,c_0)\,.
$$
\end{theorem}

From now on, for every coloring $c\in Col_{c_0}(T)$ we denote by $w_c$ the scalar $\dis{\prod_{e\in T^1_0}\dim(c_0(e))^{1/2}\prod_{e\in T^1\backslash T^1_0}\dim(c(e))}$.

\section{The homotopical Turaev-Viro invariant}\label{sec:TVhom}

In this Section, we will extend the construction of the homotopical Turaev-Viro invariant defined in \cite{HQFTp}. More precisely, the homotopical Turaev-Viro invariant was built using the graduator of a spherical category. Furthermore we will  define an homotopical Turaev-Viro invariant for every graduation of a spherical category with an invertible dimension.

\subsection{$G$-colorings}

Throughout this paragraph $\C$ will be a finitely semisimple tensor $\Bbbk$-category and $G$ will be a group.

Let $T$ be a simplicial complex. A \emph{$G$-coloring $c$} of $T$ is a map :

\begin{align*}
c:T^1_o & \fd G\\
e & \ap c(e),
\end{align*}
satisfying :
\begin{itemize}
\item[(i)]for every  oriented 1-simplex $(x_1x_2)$ of $T$:~$c(x_1x_2)=c(x_2x_1)^{-1}$\, ,
\item[(ii)]for every oriented 2-simplex  $(x_1x_2x_3)$ of $T$:~$c(x_1x_2)c(x_2x_3)c(x_3x_1)=1$\, .\\
\end{itemize}

We denote by $Col_G(T)$ the set of $G$-colorings of $T$.

In \cite{HQFTp}, we define an action on the set of $G$-colorings of $T$ using the gauge group of $T$. \emph{A gauge of $T$ with value in $G$} is a map $\delta : T^0\fd G$ and we denote $\m{G}^G_T$ the gauge group of T with value in $G$. The action is defined in the following way:
\begin{align}
\m{G}^G_T \times Col_G(T) & \fd Col_G(T) \label{action:jauge}\\
(\delta,c) & \ap c^{\delta},\nonumber
\end{align}
where $c^{\delta}$ is the $G$-coloring : $c^{\delta}(xy)=\delta(x)c(xy)\delta(y)^{-1}$, for every oriented 1-simplex $(xy)$. We denote by $Col_G(T)/\m{G}^G_T$ the quotient set of $Col_G(T)$ by the action of the gauge group $\m{G}^G_T$. We have the following topological interpretation of $Col_G(T)/\m{G}^G_T$:
\begin{proposition}[\cite{HQFTp}]\label{pro:jau}
Let $T$ be a simplicial complex, $\C$ be a semisimple tensor $\Bbbk$-category and $G$ be a group. The map :
\begin{align*}
Col_G(T) & \fd Fun(\pi_1(T),G)\\
c& \ap F_c \, ,
\end{align*}
where $F_c$ is the functor which sends every 0-simplex of $T$ to the unique object of $G$ and sends every oriented 1-simplex $(xy)$ to $c(xy)$, induces the following isomorphism :
\begin{equation}
Col_G(T)/\m{G}^G_T\cong Fun(\pi_1(T),G)/(iso)\cong [\cs{T},BG]\, ,\label{pro:suitebij}
\end{equation}
where $[\cs{T},BG]$ is the set of homotopy classes of continuous maps from the topological space $\cs{T}$ to the classifying space $BG$.
\end{proposition}

\medskip

Let us recall the topological interpretation of the $G$-colorings, in the case of manifolds with boundary.

Let $M$ be a 3-manifold, $\Sigma$ be the boundary of $M$ and $T_0$ be a triangulation of $\Sigma$. We set $Col_{G,c_0}(T)$ the set of $G$-colorings of $T$ such that the restriction to $T_0$ is $c_0$. In this case we consider the gauge action which does not change the $G$-coloring on the boundary, i.e. the restriction of $c^{\delta}$ to $T_0$ is $c_0$.

For every functor $F_0 : \pi_1(T_0)\fd G$, $Fun(\pi_1(T),G)_{F_0}$ is the set of functors $F$ from $\pi_1(T)$ to the groupoid $G$ such that the diagram :
$$
\xymatrix{
\pi_1(T) \ar[r]^{F}&  G \\
\pi_1(T_0) \ar@{^{(}->}[u]^{i} \ar[ur]_{F_0}\, ,&&
}
$$
commutes, with $i$ is the inclusion functor. We denote by $Fun(\pi_1(T),G)_{F_0}/(iso)$ the set of isomorphisms classes of functors such that the restriction of the natural isomorphisms to $\pi_1(T_0)$  is $\id_{F_0}$.

\begin{proposition}(\cite{HQFTp})\label{pro:colbound}
Let $\C$ be a semisimple tensor $\Bbbk$-category, $T$ be a simplicial complex and $T_0$ be a subcomplex of $T$. For every coloring $c_0\in Col(T_0)$, the map:
\begin{align*}
Col_{G,c_0}(T)&\fd Fun(\pi_1(T),G)_{F_{c_0}} \\
c & \ap F_c,
\end{align*}
where the functor $F_c$ sends every 0-simplex of $T$ to the unique object of the groupoid $G$ and every oriented 1-simplex $(xy)$ to $c(xy)$, induces the following isomorphism :
\begin{equation}
Col_{G,c_0}(T)/\m{G}^G_T\simeq Fun(\pi_1(T),G)_{F_{c_0}}/(iso)\, .
\end{equation}
\end{proposition}

\medskip

From now on, $\C$ is a spherical category and $(G,p)$ is a graduation on $\C$.

\medskip

Let us introduce some notations.  Let $M$ be a 3-manifold and $T$ be a triangulation of $M$. By definition of the graduation,  for every coloring $c\in Col(T)$, $pc$ is a $G$-coloring of $T$. Then for every $x\in [M,BG]$, we denote by $Col_{(G,p),x}(T)$ the set of colorings $c$ of $T$ such that the equivalence class $[pc]$ in $Col_{G}(T)/\m{G}^G_T$ corresponds to $x$  (bijection \ref{pro:suitebij}). We obtain a partition of the set $Col(T)$: $\dis{Col(T)=\coprod_{x\in [M,BG]}Col_{(G,p),x}(T)}$. If $c\in Col(T)$, we denote by $x_c\in [M,BG]$ the homotopy class associated to $pc$.

Let $M$ be a 3-manifold,  $\Sigma$ be the boundary of $M$ and  $T_0$ be a triangulation of $\Sigma$. For every homotopy class $x_0\in [\Sigma,BG]$, we denote by $[M,BG]_{\Sigma,x_0}$ the set of homotopy classes of maps from $M$ to the classifying space $BG$ such that the homotopy class of the restriction to $\Sigma$ is $x_0$. Thus for every coloring $c_0\in Col(T_0)$ and for every triangulation $T$ of $M$ such that the restriction to $\Sigma$ is $T_0$, we have:
\begin{equation}\label{isobord}
Col_{G,c_0}/(\m{G}^G_T) \cong \dis{Fun(\pi_1(T),G)_{F_{c_0}}/(iso)  \cong [M,BG]_{\Sigma,x_{c_0}}}\, .
\end{equation}

For every coloring $c_0\in Col(T_0)$ and for every homotopy class $y\in [M,BG]_{\Sigma,x_{c_0}}$, we denote by $Col_{(G,p),c_0,y}(T)$ the set of colorings $c\in Col(T)$ satisfying :
\begin{itemize}
\item $c_{T_0}=c_0$,
\item  the equivalent class $[pc]\in Col_{G,c_0}/\m{G}^G_T$ corresponds to  $y\in [M,BG]_{\Sigma,x_{c_0}}$ by the bijections (\ref{isobord}).
\end{itemize}

\medskip

Let us define the homotopical Turaev-Viro invariant obtained from the graduation $(G,p)$.

 Let $M$ be a 3-manifold, $\Sigma$ be the boundary of $M$, $T_0$ be a triangulation of $\Sigma$ and $c_0\in Col(T_0)$. We can break up the Turaev-Viro state sum in the following way:
\begin{align*}
TV_{\C}(M,c_0)&= \dc^{-n_0(T)+n_0(T_0)/2}\sum_{c \in Col_{c_0}(T)} w_c W_c \\
&= \dc^{-n_0(T)+n_0(T_0)/2} \sum_{x\in [M,BG]_{(\Sigma,x_{c_0})}}\sum_{c\in Col_{(G,p),c_0,x}(T)}w_cW_c\, ,
\end{align*}
we set: $\dis{HTV^{(G,p)}_{\C}(M,x,c_0)=\dc^{-n_0(T)+n_0(T_0)/2}\sum_{c\in Col_{(G,p),c_0,x}}w_c W_c}$.

The vector $HTV^{(G,p)}_{\C}(M,x,c_0)$ is an invariant for the triple $(M,x,c_0)$. The proof of the invariance is similar to the proof given in \cite{HQFTp} (Theorem 4.6).

\begin{theorem}\label{thm:inv}
Let $\C$ be a spherical category with an invertible dimension, $M$ be 3-manifold, $\Sigma$ be the boundary of $M$ and $T_0$ be a triangulation of $\Sigma$. For every coloring $c_0\in Col(T_0)$ and for every homotopy class $x\in [M,BG]_{\Sigma,x_{c_0}}$, where $x_{c_0}\in [\Sigma,BG]$ is obtained from $c_0$, the vector :
$$HTV^{(G,p)}_{\C}(M,c_0,x)=\dc^{-n_0(T)+n_0(T_0)/2}\sum_{c\in Col_{c_0,x}(T)}w_cW_c \in V_{\C}(\Sigma,T_0,c_0)\,$$
is an invariant of the triple $(M,x,c_0)$. We have the following equality:
\begin{equation}\label{TVsplit}
TV_{\C}(M,c_0)=\sum_{x\in [M,BG]_{\Sigma,x_{c_0}}}HTV^{(G,p)}_{\C}(M,c_0,x)\, .
\end{equation}
\end{theorem}

The vector $HTV^{(G,p)}_{\C}$ is \emph{the $(G,p)$-homotopical Turaev-Viro invariant}. The homotopical invariant defined in \cite{HQFTp} is the $(\grad,\cs{?})$-homotopical Turaev-Viro invariant.

\section{Maximal decomposition of the Turaev-Viro invariant}\label{sec:TVhommax}

Every graduation of a spherical category defines an homotopical Turaev-Viro invariant and a splitting of the Turaev-Viro invariant. We will compare these homotopical invariants. Throughout this Section, $\C$ will be a spherical category.

Let $(G,p)$ and $(H,q)$ be two graduations of $\C$. \emph{A morphism of graduation $f$ from $(G,p)$ to $(H,q)$} is a group morphism $f: G\fd H$ such that the diagram:
$$
\xymatrix{\lc \ar[rr]^{p} \ar[dr]_{q} && G \ar[dl]^{f}\commutatif\\
&H& }
$$
commutes. Notice that in the category of graduations of $\C$, where objects are the graduations of $\C$ and morphisms are the morphisms of graduation, the universal graduation is the unique initial object (up to isomorphism).

\begin{lemma}\label{lem:mapgrad}
Let $T$ be a simplicial complex, $\C$ be a finitely semisimple tensor category, $(G,p)$ and $(H,q)$ be two graduations of $\C$ and $f: (G,p)\fd (H,q)$ be a morphism of graduation. The morphism of graduation $f$ induces the following map:
\begin{align}
\ov{F} : Col_{G}(T)/\m{G}^{G}_T &\fd Col_{H}(T)/\m{G}^{H}_T \label{mapgrad}\\
[c]&\ap [f\circ c]\,.\nonumber
\end{align}
\end{lemma}
\begin{proof}
Let us show that the map (\ref{mapgrad}) is well defined. First since $f : G\fd H$ is a group morphism then for every $G$-coloring $c$, $fc$ is a $H$-coloring. Let us show that $\ov{F}$ does not depend on the choice of the representative. Let $c\in Col_{G}(T)$ and $\delta\in \m{G}^G_T$, for every oriented 1-simplex $(xy)$, one gets:
\begin{align*}
f(c^{\delta})(xy)&=f(\delta(x)c(xy)\delta(y)^{-1})\\
&= f\delta(x)fc(xy)(f\delta(y))^{-1}\\
&= (fc)^{f\delta}(xy)\, .
\end{align*}
Thus the map $\ov{F}$ is well defined.
\end{proof}

Lemma \ref{lem:mapgrad} asserts that if there is a group morphism between two graduation then we can relate the set of colorings (up to gauge actions) since the homotopical Turaev-Viro invariants are state-sum invariants indexed by the set of colorings, we can relate those invariants.

\begin{theorem}\label{thm:decompmorph}
Let $\C$ be a spherical category with an invertible dimension, $M$ be a 3-manifold, $\Sigma$ be the boundary of $M$, $T_0$ be a triangulation of $\Sigma$. For every graduations  $(G,p)$ and $(H,q)$   of $\C$ such that there exists a morphism of graduation $f : (G,p)\fd (H,q)$, we have:
\begin{equation}
HTV^{(H,q)}_{\C}(M,x,c_0)=\sum_{y\in \ov{F}^{-1}(x)}HTV_{\C}^{(G,p)}(M,y,c_0)
\end{equation}
where $F : [M,BG]\fd [M,BH]$ is the map induced by $f$ (Lemma \ref{lem:mapgrad}).
\end{theorem}
\begin{proof}
Let us recall that for every coloring $c_0\in Col(T_0)$ and for every homotopy class $x\in [M,BH]_{\Sigma,x_0}$ where $x_0\in [\Sigma,BH]$ is the homotopy class obtained from $c_0$ the vector $HTV^{(H,q)}_{\C}(M,x,c_0)$ is the state sum:
$$
HTV^{(H,q)}_{\C}(M,x,c_0)=\dc^{-n_0(T)+n_0(T_0)/2}\sum_{c\in Col_{(H,q),c_0,x}}w_c W_c
$$
Using Lemma \ref{lem:mapgrad}, we have the map:
\begin{align*}
\ov{F} : Col_G(T)/\m{G}^G_T&\fd Col_H(T)/\m{G}^H_T\\
[c] &\ap [fc]\, ,
 \end{align*}
 the map $\ov{F}$ induces a map $F : [M:BG]\fd [M,BH]$ (Proposition \ref{pro:jau}). It follows that for every $c\in Col_{(H,q),x,c_0}(T)$ we have : $c\in Col_{c_0}(T)$ and $\ov{F}([pc])=[fpc]=[qc]$ thus the homotopy class $y\in [M,BG]$ defined by $[pc]$ belongs to the set $F^{-1}(x)$. We have shown that : $\dis{Col_{(H,q),x,c_0}(T)\subset \coprod_{y\in F^{-1}(x)}Col_{(G,p),y,c_0}(T)}$. Let us show that for every $y\in F^{-1}(x)$ and for every $c_0\in Col(T_0)$ we have : $Col_{(G,p),y,c_0}(T)\subset Col_{(H,q),x,c_0}(T)$. Let $c\in Col_{(G,p),y,c_0}(T)$, it follows that $c\in Col_{c_0}(T)$ and $[qc]=[fpc]=\ov{F}([pc])$, since $y\in F^{-1}(x)$ one gets that the homotopy classes defined from the class $[qc]$ is $x$. It follows:
\begin{align*}
HTV_{C}^{(H,q)}(M,x,c_0)&=\dc^{-n_0(T)+n_0(T_0)/2}\sum_{c\in Col_{(H,q),c_0,x}}w_c W_c \\
&=\dc^{-n_0(T)+n_0(T_0)/2}\sum_{y\in F^{-1}(x)}\sum_{c\in Col_{(G,p),c_0,y}}w_c W_c \\
&= \sum_{y\in F^{-1}(x)}HTV_{\C}^{(G,p)}(M,y,c_0)\, .
\end{align*}
\end{proof}
Notice that if we consider the trivial graduation we obtain the Turaev-Viro invariant.

By definition of the universal graduation and using Theorem \ref{thm:decompmorph}, we can conclude that the splitting given by $HTV_{\C}^{(\grad,\cs{?})}$ is maximal.
\begin{corollary}\label{cor:maxHTV}
Let $\C$ be a spherical category with an invertible dimension, $M$ be a 3-manifold, $\Sigma$ be the boundary of $M$ and $T_0$ be a triangulation of $\Sigma$. For every graduation $(G,p)$ of $\C$, one gets:
$$
TV_{C}(M,c_0)=\sum_{x\in [M,BG]}HTV_{\C}^{(G,p)}(M,x,c_0)\in V_{\C}(\Sigma,T_0,c_0)\,
$$
with $c_0\in Col(T_0)$, and
$$
HTV_{\C}^{(G,p)}(M,x,c_0)=\sum_{y\in F^{-1}(x)}HTV^{(\grad,\cs{?})}(M,y,c_0)\,,
$$
where $F$ is the map induced by the universal graduation $(\grad,\cs{?})$.
\end{corollary}

\subsection*{Example}

Lens spaces $L(p,q)$, with 0<q<p and (p,q)=1, are oriented compact 3-manifolds, which result from identifying on the sphere $S^3=\{(x,y)\in \Cc^2\mid \cs{x}^2+\cs{y}^2=1\}$  the points which belong to the same orbit under the action of the cyclic group $\Zz_p$ defined by $(x,y)\ap (wx,w^qy)$ with $w=exp(2i\pi/p)$.

A singular triangulation of $L(p,q)$ is obtained by gluing together $p$ tetrahedra $(a_i,b_i,c_i,d_i)$, $i=0,...,p-1$ according to the following identification of faces $(i+1$ and $i+q$ are understood modulo p):
\begin{eqnarray}
(a_i,b_i,c_i)&=&(a_{i+1},b_{i+1},c_{i+1}) \label{i+1}\\
(a_i,b_i,c_i)&=&(b_{i+q},c_{i+q},d_{i+q}) \label{i+q}
\end{eqnarray}
The identification of (\ref{i+1}) can be realized by embedding the $p$ tetrahedra in Euclidean three-space, leading to a prismatic solid with $p+2$ 0-simplexes $a,b,c_i$, $2p$ external faces, $3p$ external edges and one internal axis $(a,b)$. Then formula (\ref{i+q}) is interpreted as the identification of the surface triangles $(a,c_i,c_{i+1})$ and $(b,c_{i+q},c_{i+1+q})$. A coloring of $L(p,q)$ is determined by the colors of the edges : $(ab)$, $(c_ic_{i+1})$ and $(bc_i)$ such that the triple is admissible. From now on, a coloring $c$ of $L(p,q)$ will be denoted by $(c(ab),c(c_ic_{i+1}),c(bc_i))$ .

In \cite{HQFTp}, we have shown that for the category of representation of $U_q(\mathfrak{sl}_2)$ with $q$ root of unity, there are two homotopical classes in $[L(p,q),B\Zz_2]$ and we have:
\begin{align}
TV_{U_q(\mathfrak{sl}_2)}(L(p,q))&=\Delta_{U_q(\mathfrak{sl}_2)}^{-2}\sum_{c=(X,Z,Y_i)}w_cW_c \nonumber\\
&= \Delta_{U_q(\mathfrak{sl}_2)}^{-2}\left(\sum_{\substack{c=(X,Z,Y_i)\\ \cs{X}=1}}w_cW_c+\sum_{\substack{c=(X,Z,Y_i)\\\cs{X}=-1}}w_cW_c\right)\, .
\end{align}
 where $HTV_0(L(p,q))$ (resp. $HTV_1(L(p,q))$) is the state sum $\dis{\Delta_{U_q(\mathfrak{sl}_2)}^{-2}\sum_{\substack{c=(X,Z,Y_i)\\ \cs{X}=1}}w_cW_c}$ (resp. $\dis{\Delta_{U_q(\mathfrak{sl}_2)}^{-2}\sum_{\substack{c=(X,Z,Y_i)\\ \cs{X}=-1,\cs{X}^p=1}}w_cW_c}$). The state sum $HTV_0$ is the homotopical Turaev-Viro invariant for the trivial homotopy class, and $HTV_1$ is the homotopical Turaev-Viro obtained for the other homotopy class.

 Let us describe the decomposition of the homotopical Turaev-Viro invariant defined for a graduation $(G,p)$ of $U_q(\frak{sl}_2)$. Using the universal property of the graduator, one get a morphism of graduation: $f : (G,p)\fd (\Zz_2,\cs{?})$. This morphism induces a map: $F : [L(p,q),B\Zz_2]\fd [L(p,q),BG]$ and Corollary \ref{cor:maxHTV} gives the following equality:
 $$
 HTV_{U_q(\frak{sl}_2)}^{(G,p)}(M,x,c_0)=\sum_{y\in F^{-1}(x)}HTV_{U_q(\frak{sl}_2)}^{(\grad,\cs{?})}(M,y,c_0)\,,
 $$

\section{The Turaev-Viro HQFT}\label{sec:TVHQFT}

In the present Section, we recall the construction of the Turaev-Viro TQFT and we will show that for every graduation of a spherical category, we can obtain a Turaev-Viro HQFT which splits the Turaev-Viro TQFT. Furthermore we will show that the splitting obtained using the universal graduation is maximal. Throughout this Section $\C$ will be a spherical category.

\subsection{The Turaev-Viro TQFT}

\subsection*{Cobordisms category}

Let $\Sigma$ and $\Sigma'$ be two oriented closed surfaces, a \emph{cobordism from $\Sigma$ to $\Sigma'$} is a 3-manifold whose boundary is the disjoint union : $\overline{\Sigma}\coprod \Sigma'$. Let $M$ and $M'$ be two cobordisms from $\Sigma$ to $\Sigma'$, $M$ and $M'$ are equivalents if there exists an isomorphism between $M$ and $M'$ such that it preserves the orientation and its restriction to the boundary is the identity.

The \emph{cobordism category} is the category where objects are closed and oriented surfaces and morphisms are equivalent classes of cobordisms. The cobordism category is denoted by $Cob_{1+2}$. The disjoint union and the empty manifold $\emptyset$ define a strict monoidal structure on $Cob_{1+2}$.

\subsection*{TQFT}
A \emph{TQFT} is a monoidal functor from the cobordism category  to the category of finite dimensional vector spaces.

\medskip

Let us recall the construction of the Turaev-Viro TQFT. Let $\Sigma$ be an oriented closed surface and $T$ be a triangulation of $\Sigma$. We associate to the pair $(\Sigma,T)$ a vector space $\dis{V_{\C}(\Sigma,T)=\bigoplus_{c\in Col(T)}\bigotimes_{f\in T_0^2}V(f,c)}$, where $V(f,c)=\ho{\Bbbi}{c(01)\pt c(12)\pt c(20)}$ for every $f=(012)$. The vector space $V(f,c)$ does not depend on the choice of a numbering which respects the orientation. Since the category $\C$ is the semi-simple, the vector space $V_{\C}(\Sigma,T)$ is dual to $V_{\C}(\ov{\Sigma},T)$, the duality is induced by the trace of the category (\cite{Tu}, \cite{GK} and \cite{HQFTp}).

Let $\Sigma$ (resp. $\Sigma'$) be an oriented surface endowed with a triangulation $T$ (resp. $T'$) and $M$ be a cobordism from $\Sigma$ to $\Sigma'$, for every colorings $c\in Col(T)$ and $c'\in Col(T')$ we have the following vector : $TV_{\C}(M,c,c')\in V_{\C}(\ov{\Sigma},T,c)\pt V_{\C}(\Sigma',T',c')\cong V_{\C}(\Sigma,T,c)^*\pt V_{\C}(\Sigma',T',c')$. The vector spaces $V_{\C}(\Sigma,T,c)$ and $V_{\C}(\Sigma',T',c')$ are finite dimensional vector spaces, thus we can build the following linear map :
$$\ov{TV_{\C}}(M)_{c,c'} : V_{\C}(\Sigma,T,c)\fd V_{\C}(\Sigma',T',c')\, ,$$
thus the matrix $\bigg( \ov{TV_{\C}}(M)_{c,c'}\bigg)_{c\in Col(T),c'\in Col(T')}$ defines the following linear map :
$$[M]= \bigg( \ov{TV_{\C}}(M)_{c,c'}\bigg)_{c\in Col(T),c'\in Col(T')}: V_{\C}(\Sigma,T)\fd V_{\C}(\Sigma',T')\, .$$
By construction of the Turaev-Viro invariant (Theorem 1.8 \cite{Tu}), we have the following relation : $[M'\cup_{\Sigma'}M]=[M']\circ[M]$ and the map $[\Sigma\times I] : V_{\C}(\Sigma,T) \fd V_{\C}(\Sigma,T)$ is an idempotent denoted by $p_{\Sigma,T}$. We set $\m{V}_{\C}(\Sigma,T)=\im(p_{\Sigma,T})$ and for every cobordism $M : \Sigma\fd \Sigma'$ we denote by $\m{V}_{\C}(M)=[M]_{\im(p_{\Sigma,T})}$ the restriction of $[M]$ to $\im(p_{\Sigma,T})$. It follows that the vector space $\m{V}_{\C}(\Sigma,T)$ is independent on the choice of the triangulation $T$. From now on, we will denote by $\m{V}_{\C}$ the Turaev-Viro TQFT, for every closed surface $\Sigma$ we denote by $\m{V}_{\C}(\Sigma)$ the vector space associated to $\Sigma$ and for every cobordism $M$ we denote by $\m{V}_{\C}(M)$ the linear map associated to $M$.

\subsection{The Turaev-Viro HQFT}

\subsection*{$B$-manifolds}

Let $B$ be a $d$-dimensional manifold, a \emph{$d$-dimensional $B$-manifold} is a pair $(X,g)$ where $X$ is closed $d$-manifold and $g : X\fd B$ is a continuous map called \emph{characteristic map}.

A \emph{$B$-cobordism from $(X,g)$ to $(Y,h)$} is a pair $(W,F)$ where $W$ is a cobordism from $X$ to $Y$ and $F$ is a relative homotopy class of a map from $W$ to $B$ such that the restriction to $X$ (resp. $Y$) is $g$ (resp. $h$). From now on, we make no notational distinction between a (relative) homotopy class  and any of its representatives.

Let $(W,F) : (M,g)\fd (N,h)$ and $(W',F') : (N',h')\fd (P,k)$ be two $B$-cobordisms and $\Psi : N\fd N'$ be a diffeomorphism such that $h'\psi=h$. The composition of  $B$-cobordisms is defined in the following way:~$(W',F')\circ(W,F)=(W'\cup W,F.F')$, where $F.F'$ is the following homotopy class :

$$
F.F'(x) = \left\{
\begin{array}{cc}
F(x) & x\in W \\
F'(x) & x\in W'\\
\end{array}
\right.
$$

Since $h'\Psi=h$, the map $F.F'$ is well defined.

The identity of $(X,g)$ is the $B$-cobordism $(X\times I,1_g)$, with $1_g$ the homotopy class of the map:
\begin{align*}
X\times I &\fd B\\
(x,t) &\ap g(x)
\end{align*}

The disjoint union of $B$-cobordisms is defined in the same way of disjoint union of cobordisms.

The \emph{category of $d+1$ $B$-cobordisms} is the category whose objects are $d$-dimensional $B$-manifolds and morphisms are isomorphism classes of $B$-cobordisms. The category of $d+1$ $B$-cobordism is denoted by $Hcob(B,d+1)$, this is a strict monoidal category.

\subsection*{HQFTs}

A \emph{$d+1$ dimensional HQFT with target space $B$} is a monoidal functor from the category $Hcob(d+1,B)$ to the category of finite dimensional vector spaces.

The vector space obtained from a $B$-manifold only depends (up to isomorphism) on the manifold and the homotopy class of the characteristic map (\cite{HQFTp}).

\subsection{The construction of the Turaev-Viro HQFT}

In \cite{HQFTp}, we have built the Turaev-Viro HQFT using the universal graduation. To build this HQFT we use the homotopical Turaev-Viro invariant $HTV^{(\grad,\cs{?})}_{\C}$. Since we have built an homotopical Turaev-Viro invariant for every graduation $(G,p)$ of a spherical category $\C$, we will obtain in the same way a Turaev-Viro HQFT. In this case the target space will the classifying space of the group of the graduation. Throughout this Section, $(G,p)$ will be a graduation $\C$.

\medskip

From now on, for every homotopy classes $x\in [\Sigma,BG]$ and $x'\in [\Sigma',BG]$ we denote by $[M,BG]_{(\Sigma,x),(\Sigma',x')}$ the set of homotopy classes of $[M,BG]$ such that the homotopy class of the restriction to $\Sigma$ (resp. $\Sigma'$) is $x$ (resp. $x'$).

For every oriented surface $\Sigma$ endowed with a triangulation $T$, we have the following decomposition:
$$V_{\C}(\Sigma,T)=\bigoplus_{x\in [\Sigma,BG]}\bigoplus_{c\in Col_{(G,p),x}(T)}V_{\C}(\Sigma,T,c)=\bigoplus_{x\in [\Sigma,BG]}V_{\C}(\Sigma,T,x) \, ,$$
with $V_{\C}(\Sigma,T,x)=\dis{\bigoplus_{c\in Col_{(G,p),x}(T)}V_{\C}(\Sigma,T,c)}$.

\medskip

Let $M$ be a cobordism from $(\Sigma,T)$ to $(\Sigma',T')$, $c$ be a coloring of $T$ and $c'$ be a coloring of $T'$. For every homotopy class $y\in [M,BG]_{(\Sigma,x_c),(\Sigma',x_{c'})}$, the vector $HTV^{(G,p)}_{\C}(M,y,c,c')\in V_{\C}(\Sigma,T,c)^*\pt V_{\C}(\Sigma',T',c')$ induces the following linear map :
$$
\ov{HTV^{(G,p)}_{\C}}(M,y,c,c') : V_{\C}(\Sigma,T,c) \fd V_{\C}(\Sigma,T',c')\, .
$$

Let $x\in [\Sigma,BG]$ and $x'\in [\Sigma,BG]$, for every $y\in [M,BG]_{(\Sigma,x),(\Sigma',x')}$ the matrix $\bigg(\ov{HTV^{(G,p)}_{\C}}(M,y,c,c')\bigg)_{c\in Col_x(T), c'\in Col_x(T')}$ defines a map from $V_{\C}(\Sigma,T,x)$ to $V_{\C}(\Sigma',T',x')$:

$$
\bigg(\ov{HTV^{(G,p)}_{\C}}(M,y,c,c')\bigg)_{c\in Col_x(T), c'\in Col_x(T')} : V_{\C}(\Sigma,T,x) \fd V_{\C}(\Sigma',T',x')\, .
$$
This map is denoted by $\widetilde{HTV^{(G,p)}_{\C}}(M,y)_{x,x'}$.

\medskip

Let $\Sigma$ be a closed and oriented surface, the inclusion $\Sigma \hookrightarrow \Sigma\times I$ is a deformation retract, thus there exists a unique homotopy class $y\in [\Sigma\times I,BG]$ such that the homotopy class of the restriction to $\Sigma\times \{0\}$ is $x$. More precisely, $y$ is the homotopy class of the following map :
\begin{align*}
\Sigma \times I &\fd BG \\
(z,t) & \ap x(z).
\end{align*}

and  we have: $[\Sigma\times I]_{(\Sigma,x),(\Sigma',x')}=\left\{\begin{array}{cc}
1_x & \mbox{if $x=x'$}\, ,\\
\emptyset & \mbox{otherwise}
\end{array}\right.$.  We denoted by $p^{(G,p)}_{\Sigma,T,x}$ the idempotent $\widetilde{HTV^{(G,p)}_{\C}}(\Sigma\times I,1_x)_{x,x}$.

\medskip

For every closed surface $\Sigma$ endowed with a triangulation $T$, we set : $\m{W}^{(G,p)}_{\C}(\Sigma,T,x)=\im(p^{(G,p)}_{\Sigma,T,x})\,.$ Let $M$ be a cobordism from $(\Sigma,T)$ to $(\Sigma',T')$, for every $x\in [\Sigma,BG]$, $x'\in [\Sigma',BG]$ and $y\in [M,BG]_{(\Sigma,x),(\Sigma',x')}$, we denote by $\m{W}^{(G,p)}_{\C}(M,y)_{x,x'}$ the restriction of $\widetilde{HTV_{\C}}(M,y)_{x,x'}$ to the vector spaces $\m{W}^{(G,p)}_{\C}(\Sigma,T,x)$ and $\m{W}^{(G,p)}_{\C}(\Sigma',T',x')$.
For every closed surface $\Sigma$ and for every triangulations $T$ and $T'$ of $\Sigma$, the linear map $\m{W}^{(G,p)}_{\C}(\Sigma\times I,1_x)_{x,x} : \m{W}^{(G,p)}_{\C}(\Sigma,T,x)\fd \m{W}^{(G,p)}_{\C}(\Sigma,T',x)$ is an isomorphism. Thus the space $\m{W}^{(G,p)}_{\C}(\Sigma,T,x)$ doesn't depend on the choice of the triangulation.

Similarly to \cite{HQFTp}, where the HQFT is obtained from $HTV_{\C}^{(\grad,\cs{?})}$, we have the following HQFT:

\begin{theorem}\label{HQFT}
Let $\C$ be a spherical category and $(G,p)$ be a graduation on $\C$. We set :
\begin{align}
\m{H}^{(G,p)}_{\C} : Hcob(BG,2+1) & \fd \textrm{vect}_{\Bbbk} \label{HQFT}\\
(\Sigma,g) & \ap \m{W}^{(G,p)}_{\C}(\Sigma,g),\nonumber\\
(M,F) & \ap  \m{W}^{(G,p)}_{\C}(M,F),\nonumber
\end{align}
where the vector space $\m{W}^{(G,p)}_{\C}(\Sigma,g)$ is defined for the homotopy class of $g$. The functor $\m{H}^{(G,p)}_{\C}$ is a $2+1$ dimensional HQFT with target space the classifying space $BG$.
\end{theorem}

To obtain the splitting of the Turaev-Viro TQFT, we will use the splitting of the idempotent which defines the Turaev-Viro TQFT \cite{HQFTp}.

\begin{lemma}\label{lem:idemspli}
Let $\C$ be  a spherical category, $(G,p)$ be a graduation  of $\C$. For every surface $\Sigma$ endowed with a triangulation $T$, we have :
\begin{equation*}
p_{\Sigma,T}=\bigoplus_{x\in [\Sigma,BG]}p^{(G,p)}_{\Sigma,T,x}.
\end{equation*}
\end{lemma}
\begin{proof}
For every 3-manifold $M$ with boundary $\Sigma$, for every triangulation $T$ of $\Sigma$ and for every coloring $c\in Col(T)$, we have : $\dis{TV_{\C}(M,c)=\sum_{x\in [M,BG]_{\Sigma,x_c}}HTV^{(G,p)}_{\C}(M,x,c)}$. if $M=\Sigma\times I$, then we have  $[\Sigma\times I]_{(\Sigma,x),(\Sigma',x')}=\left\{\begin{array}{cc}
1_x & \mbox{if $x=x'$}\, ,\\
\emptyset & \mbox{otherwise}
\end{array}\right.$. It follows that if $c,c'\in Col^{(G,p)}_x(T)$ then $TV_{\C}(\Sigma\times I,c,c')=HTV^{(G,p)}_{\C}(\Sigma\times I,1_{x},c,c')$ and if $c\in Col^{(G,p)}_x(T)$ and $c'\in Col^{(G,p)}_{x'}(T)$ with $x\not=x'$ then $TV_{\C}(\Sigma\times I,c,c')=0$. one gets  $\dis{p_{\Sigma,T}=\bigoplus_{x\in [\Sigma, BG]}p^{(G,p)}_{\Sigma,T,x}}$.
\end{proof}

Using Lemma \ref{lem:idemspli} and Theorem \ref{TVsplit}, one gets that for every graduation of a spherical category gives a decompositions of the Turaev-Viro TQFT and the blocks come from an HQFT, whose target space is given by the classifying space of the graduation.

\begin{theorem}\label{thm:splitTV}
Let $\C$ be a spherical category with an invertible dimension and $(G,p)$ be a graduation of $\C$. The Turaev-Viro TQFT $\m{V}_{\C}$ is obtained from the HQFT $\m{H}^{(G,p)}_{\C}$  :

$$
\m{V}_{\C}(\Sigma)=\bigoplus_{x\in [\Sigma,BG]}\m{H}_{\C}(\Sigma,x).
$$

For every cobordism $M : \Sigma_0 \fd \Sigma_1$ and for every $x_0\in [\Sigma_0,BG]$, $x_1\in [\Sigma_1,BG]$, we denote by $\m{V}_{\C}(M)_{x_0,x_1}$ the following restriction of the map $\m{V}_{\C}(M)$:
$$
\xymatrix{
\m{V}_{\C}(\Sigma_0) \ar[r]^{\m{V}_{\C}(M)} & \m{V}_{\C}(\Sigma_1) \\
\m{V}_{\C}(\Sigma_0,x_0) \ar@{^{(}->}[u]  \ar[r]_{\m{V}_{\C}(M)_{x_0,x_1}} & \m{V}_{\C}(\Sigma_1,x_1) \ar@{^{(}->}[u] \, .
}
$$
We have the following splitting:
\begin{equation}
\m{V}_{\C}(M)_{x_0,x_1}= \bigoplus_{y\in [M,BG]_{(\Sigma_0,x_0),(\Sigma_1,x_1)}}\m{H}_{\C}(M,y)_{x_0,x_1},\label{scindage:mo}
\end{equation}
\end{theorem}

\section{Maximal decomposition of the Turaev-Viro TQFT}\label{sec:TVHQFTmax}

In  this Section we will compare the different decompositions of the Turaev-Viro TQFT. The decomposition  obtained from the universal graduation will be the  maximal decomposition.

Let $\C$ be a spherical category, $(G,p)$ be a graduation on $\C$, $f: \grad \fd G$ the group morphism obtained form the universal property of the graduator (Proposition \ref{pro:graduation}) and  $\Sigma$ be a closed and oriented surface endowed with a triangulation $T$. For every homotopy class $x\in [M,BG]$ the vector space $\m{V}_{\C}^{(G,p)}(\Sigma)$ is the image of the idempotent $p_{\Sigma,T,x}^{(G,p)}$ and this idempotent is obtained form the vector $HTV^{(G,p)}_{\C}(\Sigma\times I,1_x)$. Using Theorem \ref{thm:decompmorph}, we have the following decomposition of the vector $HTV^{(G,p)}_{\C}(\Sigma\times I,1_x)$, for every $x\in [\Sigma, BG]$ and for every $c_0\in Col^{(G,p)}_x(T_0)$ we have:
\[
HTV_{\C}^{(G,p)}(\Sigma\times I,c_0,1_x)=\sum_{y\in F^{-1}(x)}HTV^{\grad}_{\C}(\Sigma\times I,c_0,y)\, ,
\]
where $F : [\Sigma\times I,B\grad]\fd [\Sigma\times I,BG]$ is the map induced by $f$ (Lemma \ref{lem:mapgrad}). We have shown that $\dis{Col_x^{(G,p)}(T_0)=\coprod_{y\in F^{-1}(1_x)}Col^{\grad}_{y}(T_0)}$, furthermore we have: $$[M,B\grad]_{(\Sigma,x),(\Sigma,x')}\left\{\begin{array}{cc}
   \emptyset  & \mbox{$x\not = x'$ }\\
   1_x &  \mbox{$x=x'$}\,.
 \end{array}\right.$$ It follows that:
\[
HTV^{(G,p)}(\Sigma\times I,c_0,1_x)=\sum_{y\in F_{\Sigma}^{-1}(x)}HTV^{\grad}(\Sigma,c_0,1_y)\, ,
\]
where $F_{\Sigma}$ is the restriction of $F$ to $\Sigma$. Let us  take the image of the induced idempotent, one gets: $$\m{V}(\Sigma,x)=\bigoplus_{y\in F_{\Sigma}^{-1}(x)}\m{V}^{(\grad,\cs{?})}(\Sigma,y)\, .$$

We obtain in the same way a decomposition of linear map defined by the HQFT. It follows:

\begin{theorem}\label{thm:maxsplit}
Let $\C$ be  a spherical category, $(G,p)$ be a graduation of $\C$. The Turaev-Viro HQFT obtained from the graduation $(G,p)$ is decomposed in the following way:
$$
\m{V}^{(G,p)}_{\C}(\Sigma,x)=\bigoplus_{y\in F^{-1}(x)}\m{V}^{\grad}_{\C}(\Sigma,y)\, ,
$$
for every closed surface $\Sigma$, and for every $x\in [\Sigma,BG]$, the map $F : [\Sigma,B\grad]\fd [\Sigma,BG]$ is the map obtained from the universal graduation (Lemma \ref{lem:mapgrad}).
\end{theorem}

\section*{Acknowledge}
The author was supported by Grant-in-Aid for JSPS Fellows \#19.07323.

\bibliographystyle{amsplain}
\bibliography{E:/these/biblio/biblioordre}
\end{document}